\documentclass[titlepage,12pt]{article} 
\usepackage{hyperref}
\usepackage{graphicx}
\usepackage{amssymb,amsthm,amsmath} 
\usepackage{mathrsfs}
\usepackage[a4paper]{geometry}
\usepackage{datetime2}
\usepackage[utf8]{inputenc}
\usepackage[italian,english]{babel}

\date{}

\newcommand{\ep}{\varepsilon}
\newcommand{\re}{\mathbb{R}}

\newcommand{\z}{\mathbb{Z}}

\newcommand{\ul}{u_{\lambda}}
\newcommand{\uxi}{\widehat{u}_\xi}

\newtheorem{thm}{Theorem}[section]

\newtheorem{rmk}[thm]{Remark}
\newtheorem{prop}[thm]{Proposition}

\newtheorem{lemma}[thm]{Lemma}

\title{The model example of wave equation with oscillating scale-invariant damping}

\author{Marina Ghisi\vspace{1ex}\\ 
{\normalsize Università degli Studi di Pisa} \\
{\normalsize Dipartimento di Matematica}\\ 
{\normalsize PISA (Italy)}\\
{\normalsize e-mail: \texttt{marina.ghisi@unipi.it}}
\and
Massimo Gobbino\vspace{1ex}\\ 
{\normalsize Università degli Studi di Pisa} \\
{\normalsize Dipartimento di Matematica}\\ 
{\normalsize PISA (Italy)}\\  
{\normalsize e-mail: \texttt{massimo.gobbino@unipi.it}}
}

\begin{document}
\maketitle

\begin{abstract}

We analyze a simple example of wave equation with a time-dependent damping term, whose coefficient decays at infinity at the scale-invariant rate and includes an oscillatory component that is integrable but not absolutely integrable.

We show that the oscillations in the damping coefficient induce a resonance effect with a fundamental solution of the elastic term, altering the energy decay rate of solutions. In particular, some solutions exhibit slower decay compared to the case without the oscillatory component.

Our proof relies on Fourier analysis and a representation of solutions in polar coordinates, reducing the problem to a detailed study of the asymptotic behavior of solutions to a family of ordinary differential equations and suitable oscillatory integrals.

\vspace{6ex}

\noindent{\bf Mathematics Subject Classification 2020 (MSC2020):} 
35L20, 35B40, 35B34, 35L90.

		
\vspace{6ex}


\end{abstract}

 
\section{Introduction}

In this paper we consider the damped wave equation
\begin{equation}
    u_{tt}(t,x)+\frac{m+r\cos(2t)}{t}\cdot u_t(t,x)-\Delta u(t,x)=0
    \qquad
    t\geq t_0,\quad x\in\re^d,
    \label{eqn:basic}
\end{equation}
where $d$ is a positive integer, and $t_0$, $m$, $r$ are positive real numbers. We consider the usual class of weak solutions
\begin{equation}
    u\in C^0\left([t_0,+\infty);H^1(\re^d)\right)\cap C^1\left([t_0,+\infty);L^2(\re^d)\right),
    \nonumber
\end{equation}
and we investigate the long-time behavior of their energy, given by
\begin{equation}
    \mathcal{E}_u(t):=\int_{\re^d}u_t(t,x)^2\,dx+\int_{\re^d}|\nabla u(t,x)|^2\,dx.
    \label{defn:energy-concrete}
\end{equation}

More precisely, we analyze the decay rate of the function defined as
\begin{equation}
    \mathcal{D}(t):=\sup\left\{\mathcal{E}_u(t):
    \text{$u$ is a solution to (\ref{eqn:basic}) and } 
    \mathcal{E}_u(t_0)+\int_{\re^d}u(t_0,x)^2\,dx\leq 1\right\},
    \label{defn:Energy-sup}
\end{equation}
which represents, for every $t\geq t_0$, the best constant for which the estimate
\begin{equation}
    \mathcal{E}_u(t)\leq
    \mathcal{D}(t)\cdot\left(\mathcal{E}_u(t_0)+\int_{\re^d}u(t_0,x)^2\,dx\right)
    \nonumber
\end{equation}
holds true for every solution $u$ to (\ref{eqn:basic}).

Equation (\ref{eqn:basic}) has been referenced multiple times over the past 20 years as the fundamental example that appeared to evade all known techniques. Indeed, the damping coefficient exhibits both of the key challenging features: a decay rate of order $1/t$ at infinity (known as the scale-invariant rate), and an oscillatory term that is integrable but not absolutely integrable at infinity.

\paragraph{\textmd{\textit{Some background}}}

The decay of solutions to damped wave equations of the form
\begin{equation}
    u_{tt}(t,x)+b(t)\cdot u_t(t,x)-\Delta u(t,x)=0
    \nonumber
\end{equation}
has been extensively studied since the 1970s, beginning with the model cases where 
\begin{equation}
    b(t)=\frac{m}{t^p}
    \qquad\text{(with $m>0$ and $p\in\re$)}.
    \nonumber
\end{equation}

With this choice, the function $\mathcal{D}(t)$ defined in (\ref{defn:Energy-sup}) exhibits the following asymptotic behavior, up to multiplicative constants (see~\cite{1976-Kyoto-Matsumura,1977-PJA-Matsumura,1980-Kyoto-Uesaka,2004-M2AS-Wirth}):
\begin{itemize}
    \item  it decays as $1/\log t$ if $p=-1$;

    \item  it decays as $1/t^{p+1}$ if $p\in(-1,1)$;

    \item  it decays as $1/t^{\min\{m,2\}}$ if $p=1$;

    \item  it does not decay to 0 if either $p>1$ or $p<-1$.

\end{itemize}

We observe that the decay rate is maximum when $b(t)=2/t$, while solutions do not decay at all both when the damping is too small, and when the damping is too large (which case is usually referred to as overdamping). We observe also that the value of the coefficient $m$ is relevant only when $p=1$, which represents the critical, and hence more delicate, case.

These results have been extended by J.~Wirth~\cite{2006-JDE-Wirth-NE,2007-JDE-Wirth-E}, who considered more general damping coefficients, and introduced the classification into effective and non-effective models. Roughly speaking, the general idea that emerges is that $b(t)=2/t$ represents a sort of threshold. 
\begin{itemize}
    \item When $b(t)\leq 2/t$ equation (\ref{eqn:basic}) maintains its hyperbolic nature. This is the so-called non-effective (or hyperbolic) regime, in which one expects an estimate of the form
    \begin{equation}
    \mathcal{D}(t)\lesssim\exp\left(-\int_{t_0}^t b(s)\,ds\right),
    \label{decay:hyp}
    \end{equation}
    and actually all solutions decay as the right-hand side of (\ref{decay:hyp}) and there is also a scattering theory to solutions to the undamped equation (see~\cite[Result~2]{2006-JDE-Wirth-NE} for the details). In this regime more damping yields more decay. 

    \item When $b(t)\geq 2/t$ equation (\ref{eqn:basic}) behaves more like a parabolic equation, and the so-called diffusion phenomenon appears (see~\cite{1976-Kyoto-Matsumura,1997-JDE-Nishihara,2011-JDE-RaduTodoYord} and~\cite[Theorem~27]{2007-JDE-Wirth-E}). This is the effective (or parabolic) regime, in which one expects an estimate of the form
    \begin{equation}
        \mathcal{D}(t)\lesssim\left[\int_{t_0}^t \frac{1}{b(s)}\,ds\right]^{-1},
        \label{decay:par}
    \end{equation}
    whose optimality is determined by the frequencies close to~0, meaning that the supremum in (\ref{defn:Energy-sup}) is realized by solutions concentrated on smaller and smaller frequencies. In this regime more damping yields less decay.
\end{itemize}

The general philosophy is that results of this type can be established whenever the coefficient $b(t)$ is ``well-behaved'', meaning that it remains on either side of the threshold $2/t$ and satisfies appropriate conditions that restrict its oscillatory behavior (see~\cite{2006-JDE-Wirth-NE,2007-JDE-Wirth-E,2008-JMAA-HirosawaWirth,2020-AsymptAn-VargasDaLuz,2023-AsymptAn-AslanEbert}).

On the other hand, in the absence of any control over oscillations, more complex phenomena can arise. For example, in~\cite{GGH-2016-SIAM} it is shown that, at least in the case of wave or beam equations in bounded domains, one can achieve more than exponential decay rates by exploiting suitable pulsating damping coefficients.

\paragraph{\textmd{\textit{Related results for scale-invariant damping}}}

Let us mention some results for the scale-invariant regime that are relevant to our presentation.

\begin{enumerate}

\item  (\cite[Example~3]{2023-AsymptAn-AslanEbert}) Solutions decay as prescribed by (\ref{decay:par}), namely as $1/t^{2}$, when
\begin{equation}
\frac{m-r}{t}\leq b(t)\leq\frac{m+r}{t}
\qquad\quad
m>2,
\quad
0<r\ll m-2.
\nonumber
\end{equation}

In this case the damping coefficient falls into the effective regime. Fast oscillations of the same order of the principal part are allowed, but their amplitude is required to be small.

\item  (\cite[Example~1]{2023-AsymptAn-AslanEbert}) Solutions decay as prescribed by (\ref{decay:hyp}), namely as $1/t^{m}$, when
\begin{equation}
\frac{m}{t}-\frac{1}{t\log^{\gamma}t}\leq b(t)\leq\frac{m}{t}+\frac{1}{t\log^{\gamma}t}
\qquad\quad
m\in(0,2),
\quad
\gamma>1.
\nonumber
\end{equation}

In this case the damping coefficient falls into the non-effective regime. Oscillations can be very fast, but they are a lower order term and, more important, this term is \emph{absolutely} integrable at infinity because of the condition $\gamma>1$.

\item  (\cite[Example~3.1]{2008-JMAA-HirosawaWirth}) Solutions decay as prescribed by (\ref{decay:hyp}), namely as $1/t^{m}$, when
\begin{equation}
b(t):=\frac{m(1+\sin(t^{\alpha}))}{t}
\qquad\quad
m\in(0,1/2),
\quad
\alpha\in(0,1).
\label{hp:slow}
\end{equation}

Again the damping coefficient falls into the non-effective regime, and actually it is far from the threshold $2/t$. Its oscillations have the same order as the principal part, and are not absolutely integrable at infinity.

\item  (\cite[Theorem~2.5]{2024-JDE-GG}) Solutions decay as prescribed by (\ref{decay:hyp}) and (\ref{decay:par}), namely as $1/t^{\min\{m,2\}}$, when
\begin{equation}
b(t):=\frac{m+r\sin(t^{\alpha}))}{t}
\qquad\quad
m>0,
\quad
|r|\leq m,
\quad
\alpha>1.
\label{hp:fast}
\end{equation}

Now the damping coefficient is not forced to lie on one side of the threshold $2/t$. Again its oscillations have the same order as the principal part, and are not absolutely integrable at infinity. For further results in the same spirit, we refer also to~\cite[Theorem~2.1 and Example~2.5]{2025-JMAA-GotoHirosawa}.

\item  (\cite[Theorem~2.4]{2024-JDE-GG}) For every pair of real numbers $m$ and $r$ satisfying $m\geq r>0$, there exists a function $\eta:[1+\infty)\to\re$ of class $C^\infty$ such that the integral
\begin{equation}
    \int_1^{+\infty}\frac{\cos(\eta(s))}{s}\,ds
    \nonumber
\end{equation}
is convergent, and nevertheless, when in (\ref{eqn:basic}) we choose the damping coefficient
\begin{equation}
b(t):=\frac{m+r\cos(\eta(t))}{t},
\label{hp:medium}
\end{equation}
the function $\mathcal{D}(t)$ decays at most as $1/t^{m-r/2}$. In particular, when $m-r/2\in(0,2)$ there are solutions that decay less than what prescribed by (\ref{decay:hyp}) or (\ref{decay:par}). The function $\eta(t)$ satisfies
\begin{equation}
    \eta(t)=t+O(\log t)
    \qquad
    \text{as }t\to +\infty,
    \nonumber
\end{equation}
and therefore the oscillations in (\ref{hp:medium}) have a frequency that corresponds approximately to the limit case $\alpha=1$ in (\ref{hp:slow}) and (\ref{hp:fast}).

\end{enumerate}

These examples seem to suggest that oscillations that are absolutely integrable do not alter the decay rate provided by either (\ref{decay:hyp}) or (\ref{decay:par}). Oscillations that are integrable, but not absolutely integrable, again do not alter the decay rate when they are ``too slow'' as in (\ref{hp:slow}) or ``too fast'' as in (\ref{hp:fast}). But when they oscillate with the ``correct frequency'', they create a resonance effect with the fundamental solutions of the elastic part, yielding a degradation of the decay rate. This effect is the opposite of what obtained in~\cite{GGH-2016-SIAM}, where the resonance was exploited in order to improve the decay rate.

\paragraph{\textmd{\textit{Our contribution}}}

The example provided in~\cite[Theorem~2.4]{2024-JDE-GG} was the first to show that oscillations can modify the expected decay rate of solutions in a scale-invariant regime. However, the construction of $\eta(t)$ given therein is somewhat implicit, leaving open the question of whether the same effect could occur with a simpler damping term as in (\ref{eqn:basic}).

In this paper we provide an affirmative answer. Our main result is the following.

\begin{thm}[Deterioration of the decay rate -- Model case]\label{thm:main}

For every choice of the positive real numbers $t_0$, $m$ and $r$, there exists a positive constant $C_1(t_0,m,r)$ such that the function (\ref{defn:Energy-sup}) for solutions to (\ref{eqn:basic}) satisfies
\begin{equation}
    \mathcal{D}(t)\geq\frac{C_1(t_0,m,r)}{t^{m-r/2}}
    \qquad
    \forall t\geq t_0.
    \label{th:main}
\end{equation}

\end{thm}

\begin{rmk}
\begin{em}

A careful inspection of the proof reveals that the result holds for any choice of $m$ and $r$, regardless of their sign. However, the case where the estimate is most significant is when $0<r<m$ and $m - r/2 \in (0,2)$. Indeed, in this range, (\ref{decay:hyp}) and (\ref{decay:par}) would predict a decay rate proportional to $1/t^{\min\{m,2\}}$, whereas (\ref{th:main}) ensures that the decay rate is necessarily worse. 

Moreover, in this range, it is possible to show that $\mathcal{D}(t)$ is actually proportional to $1/t^{m-r/2}$ (see also~\cite[Remark~2.7]{2024-JDE-GG}). 

\end{em}

\end{rmk}

In our main result we focused on the basic example for the sake of simplicity. However, the result is more general, and applies to abstract wave equations of the form
\begin{equation}
    u''(t)+b(t)u'(t)+Au(t)=0
    \label{eqn:abstract}
\end{equation}
where $A$ is a nonnegative multiplication operator in some Hilbert space $H$, and the damping coefficient is of the form
\begin{equation}
    b(t):=\frac{\beta(t)}{t}
    \label{defn:b-gamma}
\end{equation}
for some periodic function $\beta(t)$. The key assumption is that $\beta(t)$ resonate with at least one of the fundamental frequencies of the elastic operator $A$. We refer to Section~\ref{sec:extensions} for the details. 

\paragraph{\textmd{\textit{Overview of the technique}}}

Let us give a sketch of the proof of Theorem~\ref{thm:main}.

By Fourier analysis, it is well known that the decay rate of solutions to (\ref{eqn:abstract}) is closely linked to the decay rate of solutions to the family of ordinary differential equations
\begin{equation}
\ul''(t) + b(t) \ul'(t) + \lambda^2 \ul(t) = 0,
\label{eqn:ode-b(t)}
\end{equation}
parametrized by the real number $\lambda\geq 0$. The precise relationship is established in Proposition~\ref{prop:ode2pde} below.

In particular, for the special case of (\ref{eqn:basic}), it suffices to consider $\lambda=1$, reducing the proof of Theorem~\ref{thm:main} to the following.

\begin{prop}[Deterioration of decay for one Fourier component]\label{prop:main}

For every choice of the positive real numbers $t_0$, $m$ and $r$, there exist a positive constant $C_2(t_0,m,r)$, and a non-zero solution $v(t)$ to the ordinary differential equation
\begin{equation}
    v''(t)+\frac{m+r\cos(2t)}{t}\cdot v'(t)+v(t)=0,
    \label{eqn:ode}
\end{equation}
such that
\begin{equation}
    v'(t)^2+v(t)^2\geq 
    \left(v'(t_0)^2+v(t_0)^2\right)\frac{C_2(t_0,m,r)}{t^{m-r/2}}
    \qquad
    \forall t\geq t_0.
    \label{th:ode-below}
\end{equation}

\end{prop}

In order to prove this result, following~\cite{2024-JDE-GG} we write the pair $(v(t),v'(t))$ in polar coordinates as
\begin{equation}
    v(t)=\rho(t)\cos(\theta(t)),
    \qquad\qquad
    v'(t)=-\rho(t)\sin(\theta(t)),
    \nonumber
\end{equation}
where $\rho$ and $\theta$ are solutions to the system
\begin{eqnarray}
\rho'(t) & = & -\rho(t)\cdot\frac{m+r\cos(2t)}{t}\cdot\sin^{2}(\theta(t)),
\label{eqn:rho}
\\[0.5ex]
\theta'(t) & = & 1-\frac{m+r\cos(2t)}{t}\cdot\frac{1}{2}\sin(2\theta(t)).
\label{eqn:theta}
\end{eqnarray}

In particular, by integrating (\ref{eqn:rho}) we obtain that
\begin{equation}
    v'(t)^2+v(t)^2=
    \rho(t)^2=
    \rho(t_0)^2\exp\left(-2\int_{t_0}^t\frac{m+r\cos(2s)}{s}\cdot\sin^{2}(\theta(s))\,ds\right)
    \label{eqn:rho-int}
\end{equation}
for every $t\geq t_0$. Now the key point is that equation (\ref{eqn:theta}) admits a solution of the form
\begin{equation}
    \theta(t)=t+\varphi(t)
    \qquad\text{with}\qquad
    \lim_{t\to +\infty}\varphi(t)=0.
    \label{th:asympt-theta}
\end{equation}

This suggests the approximation
\begin{equation}
    \sin^2(\theta(t))=
    \frac{1}{2}-\frac{1}{2}\cos(2\theta(t))\sim
    \frac{1}{2}-\frac{1}{2}\cos(2t),
    \label{th:approx-sin2}
\end{equation}
so that with some simple trigonometry we obtain that
\begin{equation}
    \frac{m+r\cos(2t)}{t}\cdot 2\sin^{2}(\theta(t))\sim
    \left(m-\frac{r}{2}\right)\frac{1}{t}+(r-m)\frac{\cos(2t)}{t}
    -\frac{r}{2}\cdot\frac{\cos(4t)}{t},
    \nonumber
\end{equation}
which allows to conclude that the argument of the exponential in (\ref{eqn:rho-int}) is equal to
\begin{equation}
    -\left(m-\frac{r}{2}\right)\log\left(\frac{t}{t_0}\right)
    -(r-m)\int_{t_0}^{t}\frac{\cos(2s)}{s}\,ds
    +\frac{r}{2}\int_{t_0}^{t}\frac{\cos(4s)}{s}\,ds.    
    \nonumber
\end{equation}

Since the last two integrals remain uniformly bounded in $t$, this suffices to establish (\ref{th:ode-below}). The core of the paper is devoted to proving (\ref{th:asympt-theta}), which we establish in Proposition~\ref{prop:theta}, and providing a rigorous justification for the approximation (\ref{th:approx-sin2}), which we address at the end of Section~\ref{sec:proofs}.

\paragraph{\textmd{\textit{Further developments}}}

The results of this paper reinforce the finding of~\cite{2024-JDE-GG} that oscillations in the damping coefficient can deteriorate the decay rate and that the underlying cause of this deterioration is resonance.

In our opinion, two questions merit further investigation. The first is to establish an upper bound for the energy of solutions when the damping coefficient exhibits large oscillations between two well-behaved coefficients (see~\cite[Open~Problem~2.3]{2024-JDE-GG}). The second is to explore intermediate scenarios between (\ref{eqn:basic}) and the slow/fast oscillations described in (\ref{hp:slow}) and (\ref{hp:fast}), for instance, by considering damping coefficients such as
\begin{equation}
    \frac{m+r\cos(t\log t)}{t}
    \qquad\qquad\textmd{or}\qquad\qquad
    \frac{m+r\cos(t/\log t)}{t}.
    \nonumber
\end{equation}

In these cases, it is conceivable that some deterioration of the decay rate still occurs, though to a lesser extent than in the present paper.

\paragraph{\textmd{\textit{Structure of the paper}}}

This paper is organized as follows. In Section~\ref{sec:prelim}, we recall the classical functional setting and state the relationship between (\ref{eqn:abstract}) and (\ref{eqn:ode-b(t)}) in terms of decay rates. This should clarify why Theorem~\ref{thm:main} follows from Proposition~\ref{prop:main}.

Section~\ref{sec:proofs} is the technical core of the paper, were we prove Proposition~\ref{prop:main}. Finally, in Section~\ref{sec:extensions}, we discuss the extent to which our result extends to more general operators and damping coefficients, thus clarifying the interplay between the damping coefficient and the frequencies associated to the elastic part in the degradation of the decay rate of solutions.


\setcounter{equation}{0}
\section{Preliminaries}\label{sec:prelim}

In this section we collect some preliminary basic material, and we deepen the connection between the abstract evolution equation (\ref{eqn:abstract}) and the family of ordinary differential equations (\ref{eqn:ode-b(t)}). The final goal is showing that Proposition~\ref{prop:main} implies Theorem\ref{thm:main}.

\paragraph{\textmd{\textit{Functional setting}}}

To formulate the problem in an abstract framework, we consider the evolution equation (\ref{eqn:abstract}) within the standard functional setting of~\cite{2024-JDE-GG}. Specifically, we work with a Hilbert space \( H \) and a linear operator \( A \) defined on \( H \) with domain \( D(A) \).   We denote by $\|v\|_H$ the norm of a vector $v\in H$.

We say that \( A \) is a \emph{nonnegative multiplication operator} if there exist a measure space \( (\mathcal{M},\mu) \), a measurable function \( \lambda:\mathcal{M}\to[0,+\infty) \), and a linear isometry \( \mathscr{F}:H\to L^{2}(\mathcal{M},\mu) \) such that  
\begin{equation}
    u\in D(A)
    \quad\Longleftrightarrow\quad
    \lambda(\xi)^{2}[\mathscr{F}u](\xi)\in L^{2}(\mathcal{M},\mu),
    \nonumber
\end{equation}
and, for all \( u\in D(A) \),  
\begin{equation}
    \left[\mathscr{F}(Au)\right](\xi) =
    \lambda(\xi)^{2}[\mathscr{F}u](\xi)
    \qquad
    \forall\xi\in\mathcal{M}.
    \nonumber
\end{equation}

Intuitively, \( \mathscr{F} \) acts as a generalized Fourier transform, identifying each vector \( u\in H \) with a function \( \widehat{u} \in L^{2}(\mathcal{M},\mu) \). Under this identification, the operator \( A \) reduces to multiplication by \( \lambda(\xi)^2 \) in \( L^{2}(\mathcal{M},\mu) \).  

The \emph{spectrum} of $A$ is the set $\operatorname{Spec}(A)$ consisting of all real numbers $\lambda\geq 0$ such that
\begin{equation}
    \mu\left(\left\{\xi\in\mathcal{M}:|\lambda(\xi)-\lambda|<s\right\}\strut\right)>0
    \qquad
    \forall s>0.
    \nonumber
\end{equation}

Equation (\ref{eqn:basic}) fits into this abstract framework by setting \( H := L^2(\re^d) \) and defining \( A \) as the negative Laplacian with domain \( H^2(\re^d) \), in which case the spectrum is the entire half-line \([0,+\infty)\). With minor modifications, this abstract setting can be extended to accommodate equations in more general open sets \( \Omega \subseteq \re^d \), potentially with boundary conditions, or equations involving elliptic operators other than the Laplacian. Naturally, the spectrum depends on both the domain and the operator, and in particular it becomes discrete in the case of the Dirichlet Laplacian on a bounded domain.

\paragraph{\textmd{\textit{Reduction to a family of ODEs}}}

Thanks to the isometry \( \mathscr{F} \), there exists a bijection between solutions of the abstract equation (\ref{eqn:abstract}) and solutions of the family of ordinary differential equations (\ref{eqn:ode-b(t)}). More precisely, a function  
\begin{equation}  
    u\in C^0\left([t_0,+\infty),D(A^{1/2})\right)\cap C^1\left([t_0,+\infty),H\strut\right)  
    \nonumber
\end{equation}  
is a (weak) solution to (\ref{eqn:abstract}) if and only if, for (almost) every \( \xi \in \mathcal{M} \), the scalar function  
\begin{equation}  
    t\mapsto\uxi(t):=[\mathscr{F}(u(t))](\xi)  
    \nonumber
\end{equation}  
satisfies the ordinary differential equation  
\begin{equation}  
    \uxi{\!''}(t)+b(t)\uxi{\!'}(t)+\lambda(\xi)^2\,\uxi(t) = 0,
    \nonumber
\end{equation}  
and the total energy  
\begin{equation}  
    \mathcal{E}_u(t):=\|u'(t)\|_H^2+\|A^{1/2}u(t)\|_H^2=
    \int_{\mathcal{M}}\left(\uxi{\!'}(t)^2
    +\lambda(\xi)^2\,\uxi(t)^2\right)d\xi  
    \label{defn:energy-abstract}
\end{equation}  
remains finite for all \( t \geq t_0 \).  

Naturally, the energy defined in (\ref{defn:energy-abstract}) coincides with the energy in (\ref{defn:energy-concrete}) when applied to equation (\ref{eqn:basic}), while (\ref{defn:Energy-sup}) now reads as
\begin{equation}
    \mathcal{D}(t):=\sup\left\{\mathcal{E}_u(t):
    \text{$u$ is a solution to (\ref{eqn:abstract}) and }
    \mathcal{E}_u(t_0)+\|u(t_0)\|_H^2\leq 1\right\}.
    \label{defn:Dt-abstract}
\end{equation}

We can introduce analogous notions for solutions to the ordinary differential equation (\ref{eqn:ode-b(t)}), even because (\ref{eqn:ode-b(t)}) corresponds to the special case of (\ref{eqn:abstract}) when $H=\re$ and $A$ is multiplication by $\lambda^2$. To introduce a specific notation, for every solution to (\ref{eqn:ode-b(t)}) we set
\begin{equation}
    \mathcal{E}^{\star}_{\lambda,\ul}(t):=\ul'(t)^2+\lambda^2\ul(t)^2
    \qquad
    \forall t\geq t_0,
    \nonumber
\end{equation}
and then we define
\begin{equation}
    \mathcal{D}^{\star}_\lambda(t):=
    \sup\left\{
    \mathcal{E}^{\star}_{\lambda,\ul}(t):
    \ul\text{ solves (\ref{eqn:ode-b(t)}) and }
    \mathcal{E}^{\star}_{\lambda,\ul}(t_0)+\ul(t_0)^2\leq 1\right\}.
    \label{defn:Dt-ode}
\end{equation}

For completeness, we note that in this definition the supremum is actually attained as a maximum, since the setting is finite-dimensional.

With these preliminaries in place, we can now formalize the relationship between the decay of solutions to (\ref{eqn:abstract}) and those of (\ref{eqn:ode-b(t)}).  

\begin{prop}[Decay rates from ODEs to PDEs]\label{prop:ode2pde}

Let $H$ be a Hilbert space, let $A$ be a nonnegative multiplication operator on $H$ with spectrum $\operatorname{Spec}(A)$, let $t_0$ be a real number, and let $b:(t_0,+\infty)\to\re$ be a measurable function that is integrable in each bounded interval. Let us consider the abstract evolution equation (\ref{eqn:abstract}), and the corresponding family of ordinary differential equations (\ref{eqn:ode-b(t)}).

Then the functions (\ref{defn:Dt-abstract}) and (\ref{defn:Dt-ode}) satisfy
\begin{equation}
    \mathcal{D}(t)=\sup\left\{
    \mathcal{D}^{\star}_\lambda(t):\lambda\in\operatorname{Spec}(A)\strut\right\}
    \qquad
    \forall t\geq t_0.
    \nonumber
\end{equation}

\end{prop}

The proof of Proposition~\ref{prop:ode2pde} above is rather classical, and therefore we limit ourselves to sketching the argument. To prove the $\leq$ inequality, it is enough to consider the ``components'' $\uxi(t)$ of any solution $u(t)$ to (\ref{eqn:abstract}), and observe that they satisfy
\begin{eqnarray*}
    \mathcal{E}^{\star}_{\lambda(\xi),\uxi}(t) & \leq &
    \mathcal{D}^{\star}_{\lambda(\xi)}(t)\left(
    \mathcal{E}^{\star}_{\lambda(\xi),\uxi}(t_0)+\uxi(t_0)^2\right)
    \\
    & \leq & 
    \sup\left\{\mathcal{D}^{\star}_\lambda(t):\lambda\in\operatorname{Spec}(A)\strut\right\}
    \left(\mathcal{E}^{\star}_{\lambda(\xi),\uxi}(t_0)+\uxi(t_0)^2\right)
\end{eqnarray*}
for every $t\geq t_0$ and $\mu$-almost every $\xi\in\mathcal{M}$. When we integrate over $\mathcal{M}$ we obtain that
\begin{equation}
     \mathcal{E}_{u}(t)\leq
    \sup\left\{\mathcal{D}^{\star}_\lambda(t):\lambda\in\operatorname{Spec}(A)\strut\right\}
    \left(\mathcal{E}_{u}(t_0)+\|u(t_0)\|_H^2\right),
    \nonumber
\end{equation}
and we conclude by taking the supremum with respect to the admissible solutions $u$.

To prove the $\geq$ inequality, we consider a value $\lambda_0 \in \operatorname{Spec}(A)$ and a solution $u_{\lambda_0}$ to (\ref{eqn:ode-b(t)}) with $\lambda = \lambda_0$, whose energy at time $t$ is close to the right-hand side. We then construct a solution to (\ref{eqn:abstract}) whose components approximate $u_{\lambda_0}$ and are concentrated in the region of $\mathcal{M}$ where $\lambda(\xi)$ is close to $\lambda_0$. For further details, we refer the reader to~\cite[Section~5]{2024-JDE-GG}, where a similar technique was employed.
    
\paragraph{\textmd{\textit{Proposition~\ref{prop:main} implies Theorem~\ref{thm:main}}}}

Let us assume that for some $t\geq t_0$ and some $\lambda\in\operatorname{Spec}(A)$ there exist a constant $\Phi(t,\lambda)$ and a nonzero solution $\ul$ to (\ref{eqn:ode-b(t)}) such that
\begin{equation}
    \ul'(t)^2+\lambda^2\ul(t)^2\geq 
    \Phi(t,\lambda)\left(\ul'(t_0)^2+\lambda^2\ul(t_0)^2\right).
    \nonumber
\end{equation}

Then in particular 
\begin{equation}
    \mathcal{E}^\star_{\lambda,\ul}(t)\geq 
    \Phi(t,\lambda)\cdot\frac{\lambda^2}{1+\lambda^2}
    \left(\mathcal{E}^\star_{\lambda,\ul}(t_0)+\ul(t_0)^2\right),
    \nonumber
\end{equation}
and hence by Proposition~\ref{prop:ode2pde}
\begin{equation}
    \mathcal{D}(t)\geq
    \mathcal{D}^{\star}_\lambda(t)\geq
    \frac{\lambda^2}{1+\lambda^2}\cdot\Phi(t,\lambda).
    \nonumber
\end{equation}

Since $\lambda=1$ belongs to the spectrum of the negative Laplacian in $\re^d$, this shows that Proposition~\ref{prop:main} implies Theorem~\ref{thm:main} with $C_1(t_0,m,r)=C_2(t_0,m,r)/2$.


\setcounter{equation}{0}
\section{Proofs}\label{sec:proofs}

In this Section we prove Proposition~\ref{prop:main}, and hence also Theorem~\ref{thm:main}. To start, we recall a result on the convergence of certain oscillatory integrals, which we will use multiple times throughout the proof.

\begin{lemma}[Oscillatory integrals]\label{lemma:osc-int}

Let $K_0$ and $t_0$ be two positive real numbers, and let $h:[t_0,+\infty)\to\re$ be a function of class $C^1$ (but absolutely continuous is enough) such that
\begin{equation}
    |h'(t)|\leq\frac{K_0}{t}
    \qquad
    \forall t\geq t_0.
    \nonumber
\end{equation}

Then for every positive integer $n$, and every positive real number $\alpha$, the integrals
\begin{equation}
    \int_{t_0}^{+\infty}\frac{\cos(nt+h(t))}{t^\alpha}\,dt
    \qquad\quad\text{and}\quad\qquad
    \int_{t_0}^{+\infty}\frac{\sin(nt+h(t))}{t^\alpha}\,dt
    \nonumber
\end{equation}
are convergent.
    
\end{lemma}

The above result is nontrivial only in the range $\alpha\in(0,1]$, where the integrals are not absolutely convergent. Lemma~\ref{lemma:osc-int} was originally proved in \cite[Lemma~3.3]{2024-JDE-GG} for the case $\alpha=1$, but the proof extends without significant modifications to every $\alpha\in(0,1]$.


The second preliminary result we need is an estimate for a scalar linear equation, stating that if the restoring term has the correct sign and the forcing term has a bounded integral, then the solutions remain bounded.

\begin{lemma}\label{lemma:zg}

Let $T$ be a real number, let $f:[T,+\infty)\to\re$ and $g:[T,+\infty)\to[0,+\infty)$ be two continuous functions, and let $z:[T,+\infty)\to\re$ be a function of class $C^1$ such that
\begin{equation}
    z'(t)=f(t)-g(t)z(t)
    \qquad
    \forall t\geq T.
    \label{hp:z-eqn}
\end{equation}

Let us assume that there exists a real number $M$ such that
\begin{equation}
    \left|\int_{t_1}^{t_2}f(t)\,dt\right|\leq M
    \qquad
    \forall(t_1,t_2)\in[T,+\infty)^2.
    \label{hp:bound-A}
\end{equation}

Then it turns out that
\begin{equation}
    |z(t)|\leq|z(T)|+M
    \qquad
    \forall t\geq T.
    \label{th:z-bound}
\end{equation}

\end{lemma}

\begin{proof}

Let us fix any $t\geq T$. If either $t=T$ or $z(t)=0$, then (\ref{th:z-bound}) is trivial. Otherwise, let us assume for the time being that $z(t)>0$, and let us set
\begin{equation}
    t_1:=\inf\{\tau\in[T,t]:z(s)\geq 0\quad\forall s\in[\tau,t]\}.
    \nonumber
\end{equation}

\begin{itemize}
    \item If $t_1=T$, then $z(s)\geq 0$ for every $s\in[T,t]$. Since $g$ is nonnegative, from (\ref{hp:z-eqn}) we deduce that $z'(s)\leq f(s)$ for every $s\in[T,t]$, and therefore
    \begin{equation}
        z(t)\leq z(T)+\int_{T}^t f(s)\,ds,
    \nonumber
    \end{equation}
    so that (\ref{th:z-bound}) follows from the triangle inequality and (\ref{hp:bound-A}).

    \item If $t_1>T$, then necessarily $z(t_1)=0$ and $z(s)\geq 0$ for every $s\in[t_1,t]$. Exploiting again that $g$ is nonnegative, from (\ref{hp:z-eqn}) we deduce that $z'(s)\leq f(s)$ for every $s\in[t_1,t]$, and therefore
    \begin{equation}
        z(t)\leq z(t_1)+\int_{t_1}^t f(s)\,ds=\int_{t_1}^t f(s)\,ds,
    \nonumber
    \end{equation}
    so that again (\ref{th:z-bound}) follows from (\ref{hp:bound-A}).
\end{itemize}

The case where $z(t)<0$ can be treated in a symmetric way, or simply deduced from the previous one by observing that the function $-z$ solves an analogous differential equation, just with forcing term $-f$ instead of $f$.
    \end{proof}


Let us consider now the ordinary differential equation
\begin{equation}
    \varphi'(t)=-\frac{r}{2}\cdot\frac{\sin(\varphi(t))}{t}
    \qquad
    t>0,
    \nonumber
\end{equation}
where $r$ is a positive parameter. It is easy to verify that this equation admits stationary solutions of the form $\varphi(t) \equiv k\pi$, where $k \in \z$. Among these, the solutions with odd $k$ are unstable, while those with even $k$ are stable and attract all solutions within the strip $(k-1)\pi < \varphi(t) < (k+1)\pi$.  

In the following result, we show that a similar asymptotic behavior persists even when an integrable forcing term is added.

\begin{lemma}\label{lemma:phi2k}

Let $r$ and $t_0$ be positive real numbers, let $a:[t_0,+\infty)\to\re$ be a continuous function, and let $\varphi:[t_0,+\infty)\to\re$ be a function of class $C^1$ such that
\begin{equation}
    \varphi'(t)=a(t)-\frac{r}{2}\cdot\frac{\sin(\varphi(t))}{t}
    \qquad
    \forall t\geq t_0.
    \label{hp:eqn-phi'}
\end{equation}

Then the following statements hold.
\begin{enumerate}
\renewcommand{\labelenumi}{(\arabic{enumi})}
    \item \emph{(Convergence to a multiple of $\pi$)}. If the integral
    \begin{equation}
        \int_{t_0}^{+\infty}a(t)\,dt
    \label{hp:int-a}
    \end{equation}
    is convergent, then there exists $k\in\z$ such that
    \begin{equation}
        \lim_{t\to +\infty}\varphi(t)=k\pi.
        \label{th:phi-lim}
    \end{equation}

    \item \emph{(Rate of convergence in the even case)}. If in addition $k$ is even, and there exists a real number $\ep_0\in(0,r/2)$ such that the integral
    \begin{equation}
        \int_{t_0}^{+\infty}t^{\ep_0}a(t)\,dt
        \label{hp:int-a-rate}
    \end{equation}
    is convergent, then
    \begin{equation}
        \sup\left\{t^{\ep_0}(\varphi(t)-k\pi):t\geq t_0\right\}<+\infty.
        \label{th:rate-ep0}
    \end{equation}
\end{enumerate}

\end{lemma}

\begin{proof}

We divide the proof into four steps.

\paragraph{\textmd{\textit{Step 1 -- Existence of the limit}}}

Let us assume by contradiction that
\begin{equation}
    I:=\liminf_{t\to +\infty}\varphi(t)<\limsup_{t\to +\infty}\varphi(t)=:S,
    \label{defn:IS}
\end{equation}
and let us choose two real numbers $\alpha$ and $\beta$ satisfying
\begin{equation}
    I<\alpha<\beta<S
    \label{defn:IabS}
\end{equation}
and such that $\sin x$ has constant sign in $[\alpha,\beta]$. Just to fix ideas, let us assume that
\begin{equation}
    \sin x\geq 0
    \qquad
    \forall x\in[\alpha,\beta],
    \label{hp:sin>=0}
\end{equation}
because the other case is symmetric. From (\ref{defn:IS}) and (\ref{defn:IabS}) we deduce that there exists a sequence of intervals $[s_n,t_n]\subseteq[t_0,+\infty)$, with $s_n\to +\infty$ and $t_n\to +\infty$, such that 
\begin{equation}
    \varphi(s_n)=\alpha,
    \qquad\qquad
    \varphi(t_n)=\beta,
    \qquad\qquad
    \alpha\leq\varphi(t)\leq\beta
    \qquad
    \forall t\in[s_n,t_n].
    \nonumber
\end{equation}

As a consequence, from (\ref{hp:eqn-phi'}) and (\ref{hp:sin>=0}) we deduce that
\begin{equation}
    \varphi'(t)\leq a(t)
    \qquad
    \forall t\in[s_n,t_n],
    \nonumber
\end{equation}
and hence 
\begin{equation}
    \int_{s_n}^{t_n}a(t)\,dt\geq
    \int_{s_n}^{t_n}\varphi'(t)\,dt=
    \varphi(t_n)-\varphi(s_n)=
    \beta-\alpha
    \nonumber
\end{equation}
for every positive integer $n$, but this contradicts the convergence of the integral (\ref{hp:int-a}).

\paragraph{\textmd{\textit{Step 2 -- The limit is real}}}

Let us assume by contradiction that
\begin{equation}
    \lim_{t\to +\infty}\varphi(t)=+\infty.
    \nonumber
\end{equation}

Then for every large enough integer $n$ there exists $[s_n,t_n]\subseteq[t_0,+\infty)$ such that
\begin{equation}
    \varphi(s_n)=2n\pi,
    \qquad\qquad
    \varphi(t_n)=(2n+1)\pi,
    \nonumber
\end{equation}
and
\begin{equation}
    2n\pi\leq\varphi(t)\leq(2n+1)\pi
    \qquad
    \forall t\in[s_n,t_n].
    \nonumber
\end{equation}

At this point, the same argument of the previous step gives that
\begin{equation}
    \int_{s_n}^{t_n}a(t)\,dt\geq
    \int_{s_n}^{t_n}\varphi'(t)\,dt=
    \varphi(t_n)-\varphi(s_n)=\pi
    \nonumber
\end{equation}
for every $n$, thus contradicting the convergence of the integral (\ref{hp:int-a}).

The case where $\varphi(t)\to -\infty$ as $t\to +\infty$ can be excluded analogously.

\paragraph{\textmd{\textit{Step 3 -- The limit is a multiple of $\pi$}}}

From the previous two steps we know that
\begin{equation}
    \lim_{t\to +\infty}\varphi(t)=\ell\in\re.
    \label{hp:lim=l}
\end{equation}

We claim that $\sin\ell=0$. Let us assume by contradiction that $\sin\ell>0$ (again the other case is symmetric). Then there exists $t_1\geq t_0$ such that
\begin{equation}
    \sin(\varphi(t))\geq\frac{1}{2}\sin\ell
    \qquad
    \forall t\geq t_1,
    \nonumber
\end{equation}
and therefore from (\ref{hp:eqn-phi'}) we deduce that
\begin{equation}
    \varphi'(t)\leq a(t)-\frac{r}{4t}\sin\ell
    \qquad
    \forall t\geq t_1,      
    \nonumber
\end{equation}
so that
\begin{equation}
    \varphi(t)\leq\varphi(t_1)
    +\int_{t_1}^t a(s)\,ds
    -\frac{r\sin\ell}{4}\log\left(\frac{t}{t_1}\right)
    \qquad
    \forall t\geq t_1.
    \nonumber
\end{equation}

Letting $t\to+\infty$, and recalling the convergence of the integral (\ref{hp:int-a}), we conclude that $\varphi(t)\to -\infty$ as $t\to +\infty$, which contradicts (\ref{hp:lim=l}).

\paragraph{\textmd{\textit{Step 4 -- Convergence rate in the even case}}}

Let us assume now that $k$ is even and the integral (\ref{hp:int-a-rate}) is convergent for some $\ep_0\in(0,r/2)$. Let us introduce the function
\begin{equation}
    z(t):=t^{\ep_0}(\varphi(t)-k\pi)
    \qquad
    \forall t\geq t_0.
    \nonumber
\end{equation}

With some computations we find that
\begin{equation}
    z'(t)=\ep_0\frac{z(t)}{t}+t^{\ep_0}a(t)-t^{\ep_0}\cdot\frac{r}{2}\cdot\frac{\sin(\varphi(t))}{t}
    \qquad
    \forall t\geq t_0.
    \label{eqn:z'}
\end{equation}

Since $k$ is even, from (\ref{th:phi-lim}) we know that
\begin{equation}
    \sin(\varphi(t))=
    \sin(\varphi(t)-k\pi)=
    (\varphi(t)-k\pi)(1-\gamma(t))
    \qquad
    \forall t\geq t_0, 
    \label{defn:gamma}
\end{equation}
for a suitable function $\gamma:[t_0,+\infty)\to\re$ such that
\begin{equation}
    \lim_{t\to +\infty}\gamma(t)=0.
    \label{eqn:lim-gamma}
\end{equation}

Plugging (\ref{defn:gamma}) into (\ref{eqn:z'}) we obtain that
\begin{equation}
    z'(t)=t^{\ep_0}a(t)
    -\frac{1}{t}\left(-\ep_0+\frac{r}{2}-\frac{r}{2}\cdot\gamma(t)\right)z(t),   
    \nonumber
\end{equation}
and from (\ref{eqn:lim-gamma}) we deduce that there exists $t_2\geq t_0$ such that
\begin{equation}
    g(t):=-\ep_0+\frac{r}{2}(1-\gamma(t))\geq 0
    \qquad
    \forall t\geq t_2.  
    \label{defn:g}
\end{equation}

Therefore, we can apply Lemma~\ref{lemma:zg} to the function $z(t)$, with $g(t)$ given by (\ref{defn:g}), $T := t_2$, and $f(t) := t^{\ep_0}a(t)$. Note that condition (\ref{hp:bound-A}) follows from the convergence of the integral in (\ref{hp:int-a-rate}). As a result, we conclude that $|z(t)|$ remains bounded for all $t \geq t_2$. Since $|z(t)|$ is also bounded in the compact interval $[t_0, t_2]$, this establishes (\ref{th:rate-ep0}).
\end{proof}

\begin{rmk}[Additional insights on the structure of  solutions]
\begin{em}

It is not required for the proof of our main result, but we could say more about the asymptotic behavior of solutions to (\ref{hp:eqn-phi'}) when the integral (\ref{hp:int-a}) is convergent. Specifically, there exists a real number $\varphi_0$ with the following property.
\begin{itemize}
    \item For every $k\in\z$, the solution to (\ref{hp:eqn-phi'}) with initial condition $\varphi(t_0)=\varphi_0+2k\pi$ tends to $(2k-1)\pi$ as $t\to+\infty$.

    \item For every $k\in\z$, all solutions to (\ref{hp:eqn-phi'}) with $\varphi_0+2k\pi<\varphi(t_0)<\varphi_0+2(k+1)\pi$ tend to $2k\pi$ as $t\to+\infty$.
    
\end{itemize}

In addition, all solutions satisfy the decay condition (\ref{th:rate-ep0}) for every $\ep_0 \in (0, r/2)$ such that the integral in (\ref{hp:int-a-rate}) converges.

\end{em}
\end{rmk}


We are now ready to show that there exists a special solutions to (\ref{eqn:theta}) such that $\theta(t)-t$ that tends to~0 as $t\to +\infty$ with a rate proportional to a positive power of $t$.

\begin{prop}[A special solution to (\ref{eqn:theta})]\label{prop:theta}

Let $t_0$, $m$, $r$ be three positive real numbers.

Then there exists at least one solution $\theta:[t_0,+\infty)\to\re$ to equation (\ref{eqn:theta}) such that
\begin{equation}
    \sup\left\{t^{\ep_0}|\theta(t)-t|:t\geq t_0\strut\right\}<+\infty
    \qquad
    \forall\ep_0\in(0,\min\{r/2,1\}).
    \label{th:theta-rate}
\end{equation}
    
\end{prop}

\begin{proof}

Let us write the solution in the form
\begin{equation}
    \theta(t)=t-\frac{1}{2}\varphi(t)
    \qquad
    \forall t\geq t_0
    \nonumber
\end{equation}
for a suitable function $\varphi:[t_0,+\infty)\to\re$. Then one can show that $\theta$ solves (\ref{eqn:theta}) if and only if $\varphi$ solves
\begin{equation}
    \varphi'(t)=\frac{m+r\cos(2t)}{t}\cdot\sin(2t-\varphi(t))
    \qquad
    \forall t\geq t_0,
    \label{eqn:phi}
\end{equation}
and therefore thesis is equivalent to showing the existence of a solution to (\ref{eqn:phi}) such that
\begin{equation}
    \sup\left\{t^{\ep_0}|\varphi(t)|:t\geq t_0\strut\right\}<+\infty
    \qquad
    \forall\ep_0\in(0,\min\{r/2,1\}).
    \nonumber
\end{equation}

\paragraph{\textmd{\textit{Step 1 -- Existence of the limit}}}

We show that for every solution to (\ref{eqn:phi}) there exists an integer $k$ such that
\begin{equation}
    \lim_{t\to +\infty}\varphi(t)=k\pi.
    \label{th:phi2k}
\end{equation}

To this end, with some trigonometry we rewrite (\ref{eqn:phi}) in the form
\begin{equation}
    \varphi'(t)=a_\varphi(t)-\frac{r}{2}\cdot\frac{\sin(\varphi(t))}{t}
    \qquad
    \forall t\geq t_0,
    \label{eqn:varphi}
\end{equation}
where
\begin{equation}
    a_\varphi(t):=\frac{m}{t}\sin(2t-\varphi(t))+\frac{r}{2t}\sin(4t-\varphi(t))
    \qquad
    \forall t\geq t_0.
    \label{defn:a-phi}
\end{equation}

From (\ref{eqn:phi}) it follows that $|\varphi'(t)|\leq (m+r)/t$, and hence from Lemma~\ref{lemma:osc-int} we deduce that the integral
\begin{equation}
    \int_{t_0}^{+\infty}t^{\ep_0}a_\varphi(t)\,dt   
    \label{th:int-a-ep0}
\end{equation}
is convergent for every $\ep_0\in[0,1)$.

At this point (\ref{th:phi2k}) follows from statement~(1) of Lemma~\ref{lemma:phi2k}.

\paragraph{\textmd{\textit{Step 2 -- Uniqueness of solutions with odd $k$ in the limit}}}

We show that, for every \emph{odd} integer $k$, the exists at most one solution to (\ref{eqn:phi}) that satisfies (\ref{th:phi2k}).

Indeed, let us assume by contradiction that there exist two distinct solutions $\varphi_1$ and $\varphi_2$ with this property. Since solutions cannot cross, we can assume, without loss of generality, that
\begin{equation}
    \varphi_1(t)>\varphi_2(t)
    \qquad
    \forall t\geq t_0.
    \label{hp:phi1>phi2}
\end{equation}

If we set
\begin{equation}
    y(t):=\frac{\varphi_1(t)-\varphi_2(t)}{2}
    \qquad
    \forall t\geq t_0,
    \nonumber
\end{equation}
then with some trigonometry we obtain that
\begin{equation}
    y'(t)=
    \left(R_{\varphi_1,\varphi_2}(t)-
    r\cos\left(\frac{\varphi_1(t)+\varphi_2(t)}{2}\right)\right)
    \frac{\sin(y(t))}{2t}
    \qquad
    \forall t\geq t_0,
    \nonumber
\end{equation}
where
\begin{equation}
    R_{\varphi_1,\varphi_2}(t):=
    -2m\cos\left(2t-\frac{\varphi_1(t)+\varphi_2(t)}{2}\right)
    -r\cos\left(4t-\frac{\varphi_1(t)+\varphi_2(t)}{2}\right).
    \label{defn:R-phi}
\end{equation}

Now we know that $\varphi_1(t)$ and $\varphi_2(t)$ tend to the same limit $k\pi$ as $t\to +\infty$, with $k$ odd, and in particular
\begin{equation}
    \lim_{t\to +\infty}\cos\left(\frac{\varphi_1(t)+\varphi_2(t)}{2}\right)=
    \cos(k\pi)=-1
    \qquad\text{and}\qquad
    \lim_{t\to +\infty}y(t)=0.
    \label{th:lim-cos-y}
\end{equation}

Taking (\ref{hp:phi1>phi2}) into account, from the latter we deduce that $\sin(y(t))>0$ for every large enough $t$, and we can write it in the form
\begin{equation}
    \sin(y(t))=y(t)\left(1-\gamma(t)\right),
    \label{eqn:sin-gamma}
\end{equation}
for a suitable function $\gamma$ such that
\begin{equation}
    \lim_{t\to+\infty}\gamma(t)=0.
    \label{eqn:lim-gamma-new}
\end{equation}

From (\ref{th:lim-cos-y}), (\ref{eqn:lim-gamma-new}), and the trivial bound $|R_{\varphi_1,\varphi_2}(t)|\leq 2m+r$ for every $t\geq t_0$, we obtain that there exists $t_1\geq t_0$ such that
\begin{equation}
    \cos\left(\frac{\varphi_1(t)+\varphi_2(t)}{2}\right)\leq -\frac{1}{2}
    \qquad\quad\text{and}\quad\qquad
    \left|\gamma(t)\left(R_{\varphi_1,\varphi_2}(t)+\frac{r}{2}\right)\right|\leq\frac{r}{4}
    \label{est:t>t1}
\end{equation}
for every $t\geq t_1$. From (\ref{eqn:sin-gamma}) and (\ref{est:t>t1}) we deduce that
\begin{equation}
    y'(t) \geq
    \left(R_{\varphi_1,\varphi_2}(t)+\frac{r}{2}\right)\cdot\frac{\sin(y(t))}{2t}\geq
    \left(R_{\varphi_1,\varphi_2}(t)+\frac{r}{4}\right)\frac{y(t)}{2t}
    \qquad
    \forall t\geq t_1.
    \nonumber
\end{equation}

Integrating this differential inequality we conclude that
\begin{equation}
    y(t)\geq y(t_1)
    \exp\left(\int_{t_1}^t\frac{R_{\varphi_1,\varphi_2}(s)}{2s}\,ds
    +\frac{r}{8}\log\left(\frac{t}{t_1}\right)\right)
    \qquad
    \forall t\geq t_1.
    \nonumber
\end{equation}

Now from Lemma~\ref{lemma:osc-int} we know that the integral in the exponential is convergent when $t\to +\infty$, and therefore this estimate implies that $y(t)\to +\infty$ as $t\to +\infty$, which contradicts the second limit in (\ref{th:lim-cos-y}).

\paragraph{\textmd{\textit{Step 3 -- Existence of solutions vanishing with the correct rate}}}

From the Step~1 we know that all solutions tend to $k\pi$, as $t\to +\infty$, for some integer $k$. From Step~2 we know that the set of solutions for which $k$ is odd is at most countable. As a consequence, the exists at least one solution (and actually infinitely many of them) for which the limit is of the form $k_0 \pi$ for some \emph{even} integer $k_0$. If $\varphi(t)$ is any such solution, it is enough to observe that $\varphi(t)-k_0\pi$ is again a solution, and tends to~0 as $t\to +\infty$.

Now we focus on this special solution, and as in Step~1 we observe that it solves (\ref{eqn:varphi}) with a forcing term $a_\varphi$ given by (\ref{defn:a-phi}) for which the integral (\ref{th:int-a-ep0}) is convergent for every $\ep_0\in(0,1)$. At this point (\ref{th:theta-rate}) follows from statement~(2) of Lemma~\ref{lemma:phi2k}.
\end{proof}


\subsubsection*{Proof of Proposition~\ref{prop:main}}

As explained at the end of the introduction, it is enough to show that equation (\ref{eqn:theta}) has at least one solution $\theta(t)$ such that
\begin{equation}
    -\frac{m+r\cos(2t)}{t}\cdot 2\sin^2(\theta(t))=
    -\left(m-\frac{r}{2}\right)\frac{1}{t}+\psi(t)
    \label{integrand}
\end{equation}
for a suitable function $\psi(t)$ whose integral over $[t_0,+\infty)$ is convergent. To this end, we observe that with some trigonometry the integrand can be written as 
\begin{eqnarray*}
 -\frac{m+r\cos(2t)}{t}\cdot 2\sin^2(\theta(t)) & = &
-\frac{m+r\cos(2t)}{t}\cdot(1-\cos(2\theta(t)))
\\[0.5ex]
& = &
-\left(m-\frac{r}{2}\right)\frac{1}{t}-\frac{r\cos(2t)}{t}
\\[0.5ex]
&  &
\mbox{}+\frac{m\cos(2\theta(t))}{t}
+\frac{r}{2}\cdot\frac{\cos(2t+2\theta(t))}{t}
\\[0.5ex]
&  &
\mbox{}+\frac{r}{2}\cdot\frac{\cos(2\theta(t)-2t)-1}{t},   
\end{eqnarray*}
and we call $\psi_1(t)$, \ldots, $\psi_4(t)$ the last four functions in the right-hand side. Comparing with (\ref{integrand}), we have to show that the integrals over $[t_0,+\infty)$ of these four functions are convergent. This is immediate in the case of $\psi_1(t)$. As for $\psi_2(t)$ and $\psi_3(t)$, we observe that every solution to (\ref{eqn:theta}) can be written in the form $\theta(t)=t+h(t)$ for a suitable function $h(t)$ such that $|h'(t)|\leq (m+r)/t$, and hence the convergence of the integrals of $\psi_2(t)$ and $\psi_3(t)$ follows from Lemma~\ref{lemma:osc-int}. As for $\psi_4(t)$, from the Lipschitz continuity of the cosine we obtain that
\begin{equation}
    \frac{|\cos(2\theta(t)-2t)-1|}{t}\leq\frac{2|\theta(t)-t|}{t}
    \qquad
    \forall t\geq t_0.
    \nonumber
\end{equation}

Therefore, if we consider the special solution to (\ref{eqn:theta}) provided by Proposition~\ref{prop:theta}, from (\ref{th:theta-rate}) we know that $|\theta(t)-t|$ decays at least as some positive power of $t$, which is enough to conclude that the integral of $\psi_4(t)$ is even absolutely convergent.
\qed

\begin{rmk}[The isolated fast solution]
\begin{em}

Although not necessary for the proof of our main result, we can further analyze the asymptotic behavior of solutions to (\ref{eqn:ode}). In particular, there exists a real number $\theta_0\in[0,2\pi)$ with the following property.
\begin{itemize}
    \item For every positive initial energy $E_0^2$, the solution to (\ref{eqn:ode}) with initial data 
    \begin{equation}
        v(t_0)=E_0\cos(\theta_0),
        \qquad\qquad
        v'(t_0)=-E_0\sin(\theta_0)
    \nonumber
    \end{equation}
    satisfies
    \begin{equation}
        \lim_{t\to +\infty}\left(v'(t)^2+v(t)^2\right)t^{m+r/2}<+\infty.
    \nonumber
    \end{equation}

    \item For every positive initial energy $E_0^2$, and for every $\theta_1\in[0,2\pi)$ with $\theta_1\neq\theta_0$, the solution to (\ref{eqn:ode}) with initial data 
    \begin{equation}
        v(t_0)=E_0\cos(\theta_1),
        \qquad\qquad
        v'(t_0)=-E_0\sin(\theta_1)
    \nonumber
    \end{equation}
    satisfies
    \begin{equation}
        \lim_{t\to +\infty}\left(v'(t)^2+v(t)^2\right)t^{m-r/2}>0.
    \nonumber
    \end{equation}

\end{itemize}

In other words, for every positive initial energy, there is exactly one solution that exhibits a ``fast decay'' of order $t^{m+r/2}$, while all remaining solutions decay ``slowly'' at a rate of $t^{m-r/2}$. The fast-decaying solution corresponds to the unstable solutions of (\ref{eqn:phi}), namely those tending to odd multiples of $\pi$, and for them resonance acts ``in favor'' of the decay rate.

\end{em}
\end{rmk}


\setcounter{equation}{0}
\section{Possible generalizations}\label{sec:extensions}

At this point, one may ask whether, or under what assumptions, the degradation of the decay rate observed in Theorem~\ref{thm:main} can also occur for more general equations of the form (\ref{eqn:abstract}). Let us discuss some possible scenarios.  

\paragraph{\textmd{\textit{More general operators}}}  

Consider the abstract equation (\ref{eqn:abstract}) with the same damping coefficient as in (\ref{eqn:basic}). Then, the conclusion of Theorem~\ref{thm:main} remains valid, provided that \( \lambda = 1 \) belongs to the spectrum of \( A \). Moreover, if \( \lambda = 1 \) is also an eigenvalue of \( A \), then we can strengthen the conclusion by stating that there exist a positive constant $C_3(t_0,m,r)$ and a nonzero solution satisfying  
\begin{equation}
    \mathcal{E}_u(t) \geq \mathcal{E}_u(t_0)\cdot\frac{C_3(t_0,m,r)}{t^{m - r/2}}   
    \qquad  
    \forall t \geq t_0.
    \nonumber
\end{equation}  

\paragraph{\textmd{\textit{Different frequencies}}}

The result also holds for damping coefficients of the form  
\begin{equation}
    b(t) = \frac{m + r\cos(2\lambda_0 t)}{t},
    \nonumber
\end{equation}
provided that \( \lambda_0 > 0 \) belongs to the spectrum of \( A \). In this case, we must consider equation (\ref{eqn:ode-b(t)}) with \( \lambda = \lambda_0 \) instead of \( \lambda = 1 \). The key point is that, when we introduce polar coordinates, the equation for $\theta$ has a solution of the form \( \theta(t) = \lambda_0 t + o(1) \), which implies that in (\ref{eqn:rho}) the term \( \sin^2(\theta(t)) \) resonates again with \( \cos(2\lambda_0 t) \).  

\paragraph{\textmd{\textit{Adding a phase}}}  

The result remains valid if we replace \( \cos(2t) \) with \( \cos(2t + \alpha_0) \) for some constant \( \alpha_0 \). Indeed, by applying a time translation, this is equivalent to considering the damping coefficient  
\begin{equation}
    b(t) = \frac{m + r\cos(2t)}{t - \alpha_0/2},
    \nonumber
\end{equation}
for which the proof remains essentially unchanged. Note that when \( \alpha_0 = 3\pi/2 \) the oscillatory term transforms into \( \sin(2t) \), and that $\alpha_0=\pi$ is equivalent to switching the sign of $r$.  

\paragraph{\textmd{\textit{Adding a non-resonant term}}}

Finally, we consider damping coefficients of the form  
\begin{equation}
    b(t) = \frac{m + r\cos(2t) + \delta(t)}{t},
    \nonumber
\end{equation}
where \( \delta \in L^1_{\text{loc}}(\mathbb{R}) \) is a periodic function with period \( k\pi \) for some positive integer \( k \), and satisfying  
\begin{equation}
    \int_0^{k\pi} \delta(t)\, dt =
    \int_0^{k\pi} \delta(t)\cos(2t)\, dt =
    \int_0^{k\pi} \delta(t)\sin(2t)\, dt =
    0.
    \label{hp:non-resonance}
\end{equation}

Following the heuristic argument presented in the introduction, these conditions suggest that the presence of \( \delta(t) \) does not significantly affect the key integrals that determine the decay rate of solutions. To make this argument rigorous, one would need to refine the proof of Proposition~\ref{prop:theta}, carefully incorporating the additional terms involving \( \delta(t) \) that appear in both the definition (\ref{defn:a-phi}) of \( a_\varphi \) and the definition (\ref{defn:R-phi}) of \( R_{\varphi_1,\varphi_2} \), as well as in the expansion of the integrand in the left-hand side of (\ref{integrand}). However, due to the non-resonance conditions in (\ref{hp:non-resonance}), all these terms remain integrable over \( [t_0,+\infty) \), rendering their contribution negligible. We leave the details to the interested reader.

\paragraph{\textmd{\textit{General periodic function}}}

Let us now consider any function \( \beta \in L^1_{\text{loc}}(\mathbb{R}) \) that is periodic with some period \( p_0 > 0 \), and let us take the damping coefficient (\ref{defn:b-gamma}). As in the model case, the integral of \( \beta \) is asymptotically equivalent to \( m \log t \), where \( m \) is the average of \( \beta \) over \( (0, p_0) \). Therefore, according to (\ref{decay:hyp}), one might expect a decay rate of solutions proportional to \( t^{-m} \), at least when \( m \in (0,2) \).  

On the other hand, if \( \beta \) is non-constant, then it has a nonzero component at some frequency that is an integer multiple of \( 2\pi/p_0 \). In particular, it can be expressed as  
\begin{equation}
    \beta(t) = m + r \cos(2\lambda_0 t + \alpha_0) + \delta(t),
    \nonumber
\end{equation}
for suitable positive real numbers \( r \), \( \lambda_0 \), \( \alpha_0 \), and a function \( \delta \) satisfying  
\begin{equation}
    \int_0^{p_0} \delta(t)\,dt =
    \int_0^{p_0} \delta(t) \cos(2\lambda_0 t)\,dt =
    \int_0^{p_0} \delta(t) \sin(2\lambda_0 t)\,dt =
    0.
    \nonumber
\end{equation}

At this point, if \( \lambda_0 \) happens to be in the spectrum of \( A \), then combining the previous observations leads to a resonance effect, implying that the decay rate is at most proportional to \( t^{-m + r/2} \). Formally, this leads to a statement of the following form.

\begin{thm}[Deterioration of the decay rate -- General case]

Let $H$ be a Hilbert space, let $A$ be a nonnegative multiplication operator on $H$, and let $\beta\in L^1_{loc}(\re)$.

Let us assume that there exists positive real numbers $\lambda_0$, and $p_0$ such that
\begin{enumerate}
\renewcommand{\labelenumi}{(\roman{enumi})}
    \item $\lambda_0$ is in the spectrum of $A$,

    \item the function $\beta$ is periodic with period $p_0$,

    \item $\lambda_0p_0$ is an integer multiple of $\pi$.

\end{enumerate}

Let us set
\begin{gather*}
    r_1:=\frac{2}{p_0}\int_0^{p_0}\beta(t)\cos(2\lambda_0 t)\,dt,
    \qquad\qquad
    r_2:=\frac{2}{p_0}\int_0^{p_0}\beta(t)\sin(2\lambda_0 t)\,dt,
    \\[1ex]
    m:=\frac{1}{p_0}\int_0^{p_0}\beta(t)\,dt,
    \qquad\qquad
    r:=\sqrt{r_1^2+r_2^2}.
\end{gather*}

Then solutions to the abstract evolution equation (\ref{eqn:abstract}), with $b(t)$ given by (\ref{defn:b-gamma}), satisfy 
\begin{equation*}
    \mathcal{D}(t)\geq\frac{C_4(t_0,\beta)}{t^{m-r/2}}
    \qquad
    \forall t\geq t_0
\end{equation*}
for a suitable positive constant $C_4(t_0,\beta)$ that depends on $t_0$ and on the function $\beta$.
    
\end{thm}


\subsubsection*{\centering Acknowledgments}

Both authors are members of the Italian {\selectlanguage{italian}%
``Gruppo Nazionale per l'Analisi Matematica, la Probabilit\`{a} e le loro Applicazioni'' (GNAMPA) of the ``Istituto Nazionale di Alta Matematica'' (INdAM)}. 

The authors acknowledge the MIUR Excellence Department Project awarded to the Department of Mathematics, University of Pisa, CUP I57G22000700001.



\label{NumeroPagine}

\end{document}